\documentclass[12 pt]{amsart}
\newtheorem{theorem}{Theorem}[section]

\newtheorem{proposition}[theorem]{Proposition}

\newtheorem{definition}[theorem]{Definition}

\theoremstyle{definition}
\newtheorem{remark}[theorem]{Remark}
\newtheorem{example}[theorem]{Example}
\numberwithin{equation}{section}
\newcommand{\CC}{\mathbb{C}}

\newcommand{\ZZ}{\mathbb{Z}}
\newcommand{\DD}{\mathbb{D}}

\newcommand{\NN}{\mathbb{N}}
\newcommand{\bV}{\mathbf{V}}
\newcommand{\bE}{\mathbf{E}}
\newcommand{\bF}{\mathbf{F}}
\newcommand{\bS}{\mathbf{S}}
\newcommand{\bP}{\mathbf{P}}
\newcommand{\cH}{\mathcal{H}}
\newcommand{\cM}{\mathcal{M}}

\DeclareMathOperator{\spa}{span}

\DeclareMathOperator{\ran}{ran}
\title[ Rational DA functions on a rhombic lattice]{Rational discrete analytic functions on a rhombic lattice}

\author[D.  Alpay]{Daniel Alpay}
\address{(DA)
Faculty of Mathematics, Physics, and Computation\\
Schmid College of Science and Technology\\
Chapman University\\
One University Drive\\
Orange, CA 92866\\
USA}
\email{alpay@chapman.edu}

\author[Z. Kazi]{Zubayir H. Kazi}
\address{(ZK)
Department of Mathematics\\
West Valley College\\
14000 Fruitvale Avenue\\
Saratoga, CA 95070\\
USA}
\email{zubairkazi10@gmail.com}

\author[M. Tecalero]{Mariana Tecalero}
\address{(MT)
College of Science, Technology, Engineering, Arts, and Mathematics\\
Alvernia University\\
400 Saint Bernardine Street\\
Reading, PA 19607\\
USA}
\email{marianatecalero@gmail.com}

\author[D.  Volok]{Dan Volok}
\address{(DV) Department of Mathematics\\ Kansas State University\\ 138 Cardwell Hall\\1228 N.  17th Street\\
  Manhattan, KS 66506\\ USA}
\email{danvolok@math.ksu.edu}

\begin{document}
\dedicatory{Dedicated to Palle Jorgensen on the occasion of his 75th birthday}
\maketitle
\begin{abstract} A novel basis of discrete analytic polynomials on a rhombic lattice is introduced and the associated convolution product is studied. A class of discrete analytic functions that are rational with respect to this product is also described.
\end{abstract}
\mbox{}\\



\date{today}
%\tableofcontents
\section{Introduction}
Discrete complex analysis has drawn much attention in the recent years; see for instance \cite{lovasz} and \cite{smirnov1}. The first advances in this field are associated with the works of J. Ferrand  and R. J. Duffin (\cite{Ferrand}, \cite{Duffin}) who studied certain discretizations of complex analytic functions on the integer lattice in the complex plane $\CC.$ From the beginning, one of the basic issues that arose in the investigation of discrete analytic (DA for short) functions was that the usual point-wise product of functions does not preserve discrete analyticity. In spite of this, DA polynomials are well enough understood and several  different  bases of DA polynomials on the integer lattice can be found in the literature. Here, in addition to the above-cited works, one should mention the work of D. Zeilberger  \cite{zeil} who first constructed a basis suitable for "power" series expansions.

The situation with rational DA functions is more complicated. In \cite{MR0052526} R. Isaacs (who used a slightly different notion of discrete analyticity) formulated a conjecture that every rational DA function is a polynomial.  Isaacs' conjecture was ultimately proved false by Harman \cite{MR0361111}.  However,  these works indicate that the class of DA functions, that are rational in the sense of the point-wise product, is not rich enough for applications. 

Another notion of DA rationality for functions on the integer lattice in $\CC$  is  based on the Cauchy - Kovalevskaya extension of the point-wise product from the real axis.  It was introduced and studied in \cite{ajsv}. This approach was inspired by  \cite{jfa}, which dealt with another case of analyticity being incompatible with the point-wise product --  hyperholomorphic functions of the quaternionic variable.

In more recent works \cite{av2022} and \cite{av2023}, the convolution product $\odot$ of DA polynomials on the integer lattice with respect to a certain "standard" basis $z^{(n)}$ was considered instead.  Although this setup proved to be more fruitful, it is heavily dependent on the properties of the elementary difference operators
$$\delta_x f(z)=f(z+1)-f(z),\quad \delta_y f(z)=f(z+i)-f(z).$$
The main goal of the present paper is to extend the notion of rationality  to the case of a rhombic lattice, where the difference operators do not seem to be very helpful. 
As it turns out,  the forward shift operator on the space of DA polynomials can be adapted to the rhombic lattice case much more easily. 

One of the main ingredients 
is that every rational matrix-valued function  $f(z),$ analytic in a neighborhood of the origin, can be  represented as
$$f(z)=D+zC(I-zA)^{-1}B,$$
where $A,B,C,D$ are matrices of suitable dimensions. Such representations, called {\em realizations}, arose in linear system theory (see \cite{kalman}) and proved quite useful in the study of rational matrix-functions (e.g. \cite{bgk1}).

The paper is organized as follows. Section \ref{prelim} provides the background on  discrete analyticity on rhombic lattices, most of which goes back to \cite{rhombic} and
\cite{kenyon}. In Section \ref{shiftsec} the forward and backward shift operators are introduced. Investigation of their properties  (Theorem \ref{fwdthm}, Proposition \ref{bwdprop}) justifies the terminology: they do resemble the shift operators in the complex analytic case; see Remark \ref{shiftrem}. Section \ref{DAP} contains the construction of the eigenfunction of the backward shift operator (Theorem \ref{genfunthm}). A power series expansion of this eigenfunction in terms of a continuous parameter generates "pseudo-powers" of $z$ that serve as a DA polynomial basis. Finally, in Section \ref{ratsec} the main definitions and results are presented. The convolution product
$\odot$ of DA polynomials is set up in Definition \ref{convoproddef}, and Theorem \ref{main} states that $\odot$ -quotients of DA polynomials are precisely those DA functions that admit realizations, thus justifying the notion of a "rational DA function." The relationship between rational DA functions and rational functions of the complex variable is formulated in Theorem \ref{ratmap}. This relationship leads to explicit formulas for products and inverses of rational DA functions in terms of realizations. A characterization of rational DA functions in terms of shift-invariance is stated as Theorem \ref{shiftcar}. Finally, an example of a rational DA function $K_w(z)$ is given (Example \ref{rk-ex}). 
This rational function is a positive reproducing kernel that gives rise to a Hilbert space of DA functions, to be investigated in a forthcoming publication.

\section{Preliminaries}\label{prelim}
In what follows, $\Lambda$ is a rhombic lattice in the complex plane $\CC$ -- a monohedral tessellation of $\CC$ with  unit rhombi. The sets of vertices, edges and faces of $\Lambda$ are denoted by $\bV(\Lambda),$ $\bE(\Lambda)$ and $\bF(\Lambda),$ respectively. A {\em track} in $\Lambda$ is a sequence $(F_n)_{n\in\ZZ}$ in $\bF(\Lambda),$ such that for every $n\in\ZZ$  faces $F_n$ and $F_{n+1}$ are adjacent and, furthermore,  the edges shared by $F_n$ with its neighbors $F_{n+1}$ and $F_{n-1}$ form two opposite sides of the rhombus $F_n.$ These shared edges are referred to as {\em ties} of the track, and the rest are called {\em rails} of the track. Note that any two ties of a track are parallel. The same cannot be said, in general, about rails  of  a track.

The set of directions of the edges of $\Lambda$ is denoted by $\hat{\bE}(\Lambda):$
$$\hat{\bE}(\Lambda)=\{b-a:(a,b)\in \bE(\Lambda)\}.$$

\begin{proposition}\label{finder}
 $\hat{\bE}(\Lambda)$ is a finite set.
\end{proposition}

\begin{proof} The case of a square lattice is trivial. If the faces of $\Lambda$ are not squares, denote by $\alpha$ the radian measure of the acute angle in each of the rhombic faces of $\Lambda.$ If $\alpha/\pi$ is a rational number, say $$\dfrac{\alpha}{\pi}=\dfrac{m}{n},\quad\text {where}\quad m,n\in\NN,$$ then connectedness of $\Lambda$ implies that for a suitable $\alpha_0\in\mathbb{R}$
$$\hat{\bE}(\Lambda)\subset\{e^{i\pi(\alpha_0+k)/n}:k=0,1,\dots,2n-1\}.$$
If $\alpha/\pi$ is irrational, then every vertex of $\Lambda$ necessarily has degree $4,$ and lattice $\Lambda$ is formed by parallel translations of a single track in the direction of its ties, much like in the square lattice case.
\end{proof}

Let $a,b\in \bV(\Lambda).$ A path in $\Lambda$ from 
$a$ to $b$ is a finite sequence $(z_0,z_1,\dots,z_N)$ of vertices of $\Lambda,$ where $z_0=a,$ $z_N=b,$ and  for $n=1,\dots,N$
the vertices $z_{n-1}$ and $z_n$ are adjacent in $\Lambda.$  A path from $a$ to $b$ is closed if $a=b.$ 

Given a function $f:\bV(\Lambda)\longrightarrow\CC,$  the discrete integral of $f$ over a path 
 $\gamma=(z_0,\dots,z_N)$ is defined by
$$\int_\gamma f \delta z=\sum_{n=1}^N \dfrac{f(z_{n-1})+f(z_n)}{2}(z_n-z_{n-1}).$$

\begin{definition}\label{ferrand} Function $f:\bV(\Lambda)\longrightarrow\CC$ is said to be discrete analytic (DA) if for every closed path $\gamma$ in $\Lambda$
$$\int_\gamma f \delta z=0.$$
\end{definition}

Note that a  closed path of the form $(a,b,c,d,a),$
where $a\not=c$ and $b\not=d$ forms a face of $\Lambda;$ it is denoted simply by $abcd.$

\begin{theorem}\label{CR}
Function $f:\bV(\Lambda)\longrightarrow\CC$ is DA if, and only if,  on every face $abcd$ of $\Lambda$ $f(z)$ satisfies the discrete Cauchy - Riemann equation
$$\dfrac{f(a)-f(c)}{a-c}=\dfrac{f(b)-f(d)}{b-d}.$$
\end{theorem}
\begin{proof} Since the integral over any closed path in $\Lambda$ can be written as a sum of integrals over faces, $f(z)$ is DA if, and only if,  for every face $abcd$ of $\Lambda$ 
$$\int_{abcd} f \delta z=0,$$ which is equivalent to \eqref{CR}.
\end{proof}

As is clear from Definition \ref{ferrand},  the integral of DA function $f(z)$ over a path $\gamma$ from $a$ to $b$ is independent  of the choice of $\gamma,$ hence a simplified notation can be used:
$$\int_a^b f \delta z=\int_\gamma f\delta z,$$
where $\gamma$ is any path in $\Lambda$ from $a$ to $b.$

\begin{theorem} Let $z_0\in V(\Lambda)$ be fixed. If $f(z)$ is a DA function, then so is
$$F(z)=\int_{z_0}^z f\delta z.$$\end{theorem}
\begin{proof} Let $abcd$ be a face of $\Lambda.$ One has
\begin{multline*}
\dfrac{F(a)-F(c)}{a-c}-\dfrac{F(b)-F(d)}{b-d}\\
=\dfrac{1}{a-c}\int_c^a f\delta z-\dfrac{1}{b-d}\int_d^b f\delta z\\
=\dfrac{f(c)(b-c)+f(b)(a-c)+f(a)(a-b)}{2(a-c)}-\\-\dfrac{f(d)(c-d)+f(c)(b-d)+f(b)(b-c)}{2(b-d)}
\\=\dfrac{(f(a)-f(c))(a-b)}{2(a-c)}+\dfrac{(f(b)-f(d))(c-d)}{2(b-d)}
\\
=\dfrac{f(a)-f(c)}{2(a-c)}(a-b+c-d)+\\
+\dfrac{c-d}{2}\left(\dfrac{f(b)-f(d)}{b-d}-\dfrac{f(a)-f(c)}{a-c}\right)=0
\end{multline*}
in view of Theorem \ref{CR} and the fact that $abcd$ is a rhombus.
\end{proof}

In what follows, two additional assumptions about the lattice $\Lambda$ are made. Firstly,  $0\in \bV(\Lambda).$
Secondly, for every $z\in \bV(\Lambda)$ there is $N\in\NN$ and a path $(z_0,\dots,z_N),$ such that:
\begin{gather}\label{good1} z_0=z;\\
z_n-z_{n-1}\not=\pm1,\quad n=1,\dots,N;\\
z_N-z_{N-1}=1.\label{good3}
\end{gather}
A path with the properties \eqref{good1} -- \eqref{good3} is said to be a {\em leash} of $z$ of length $N.$ 

\begin{remark}
The above assumptions can be satisfied with a suitable affine linear transformation. Indeed,
in view of Proposition \ref{finder}, one can find an infinite family of tracks in $\Lambda$ that that have parallel ties. With a suitable  rotation,  one can ensure that  these  ties  are parallel to the real axis  and, moreover, that every vertex in $\Lambda$ lies to the left of infinitely many tracks in the family.
\end{remark}

\section{Shift operators on the space of DA functions}
\label{shiftsec}
Denote by $\cH(\lambda)$ the space of DA functions on $\Lambda.$ 
\begin{definition}\label{fwd}
The forward shift operator $Z_+:\cH(\Lambda)\longrightarrow \cH(\Lambda)$ is defined by
$$(Z_+f)(z)=\dfrac{f(0)-f(z)}{2}+\int_0^z f\delta z.$$
\end{definition}

 \begin{theorem} \label{fwdthm} The kernel and range of $Z_+$ can be characterized as follows:
$$\ker(Z_+)=\{0\},\quad
\ran(Z_+)=\{f\in \cH(\Lambda) : f(0)=0\}.
$$
\end{theorem}
\begin{proof} Suppose first that $f\in\cH(\Lambda)$ is such that $Z_+f=0.$ Then
$$\dfrac{f(z)-f(0)}{2}=\int_0^z f\delta z,\quad z\in \bV(\Lambda).$$
Therefore,  if $a,b\in \bV(\Lambda),$ then
$$f(b)-f(a)=2\int_a^b f\delta z.$$
In particular, if the vertices $a$ and $b$ are adjacent, then
\begin{equation}\label{re2a}\dfrac{f(b)-f(a)}{b-a}=f(a)+f(b).\end{equation}
Now fix an arbitrary  $z\in \bV(\Lambda),$ and let $(z_0,\dots,z_N)$ be a leash of $z.$ Then \eqref{re2a} implies that
$$f(z)=f(z_1)\dfrac{1+z_0-z_1}{1+z_1-z_0}=\dots=f(z_N)\prod_{n=1}^N\dfrac{1+z_{n-1}-z_n}{1+z_n-z_{n-1}}=0,$$
since $1+z_{N-1}-z_N=0.$ Thus $\ker(Z_+)=\{0\}.$

As to the range of $Z_+,$ the inclusion 
 $$\ran(Z_+)\subset\{f\in \cH(\Lambda) : f(0)=0\}
$$
follows immediately from Definition \ref{fwd}. To prove the opposite inclusion,
 assume that $g(z)$ is a DA function, such that $g(0)=0.$ One needs to show that there exists a DA function $f(z),$
such that $Z_+f=g.$ This can be done in a number of steps.
\noindent\paragraph{\bf Step 1.}  Let $f:\bV(\Lambda)\longrightarrow\CC$ be given. Then $f(z)$ is the pre-image of $g(z)$ under $Z_+$ if, and only if,
   the relation
\begin{equation}\label{aux524}f(v)(1+u-v)-f(u)(1+v-u)=2(g(u)-g(v))\end{equation}
holds on  every edge $(u,v)$ of $\Lambda.$ Indeed, if $f\in\cH(\Lambda)$ is such that $Z_+f=g,$ then for every pair $u,v\in\bV(\Lambda)$ it holds that
$$g(u)-g(v)=Z_+f(u)-Z_+f(v)=\dfrac{f(v)-f(u)}{2}+\int_v^u f\delta z.$$
In particular, if the vertices $u$ and $v$ are adjacent, then
$$g(u)-g(v)=\dfrac{f(v)-f(u)}{2}+\dfrac{f(v)+f(u)}{2}(u-v),$$
which is equivalent to \eqref{aux524}. Conversely, if \eqref{aux524} holds on every edge of $\Lambda$, then on every face $a_1a_2a_3a_4$ of $\Lambda$
one has (with $a_0:=a_4$):
\begin{multline*}\int_{a_1a_2a_3a_4}f\delta z=\sum_{k=1}^4\dfrac{f(a_{k-1})+f(a_k)}{2}(a_k-a_{k-1})\\
=\sum_{k=1}^4\left(g(a_k)-g(a_{k-1})-\dfrac{f(a_{k-1})-f(a_k)}{2}\right)
=0,\end{multline*}
hence $f(z)$ is a DA function. Now fix an arbitrary $z\in\bV(\Lambda)$ and choose a path $(z_0,\dots z_N)$ from $z_0=0$ to $z_N=z.$
Applying \eqref{aux524} to every edge of the path, one gets
\begin{multline*}(Zf)(z)=\dfrac{f(0)-f(z)}{2}+\int_0^z f\delta z\\=\sum_{n=1}^N\dfrac{f(z_{n-1})-f(z_n)}{2}+\sum_{n=1}^N\dfrac{f(z_{n-1}+f(z_n)}{2}(z_n-z_{n-1})\\
=\sum_{n=1}^N(g(z_n)-g(z_{n-1}))=g(z)-g(0)=g(z).
\end{multline*}

\noindent\paragraph{\bf Step 2.} Whenever $z\in\bV(\Lambda)$ has a leash of length 1, 
set
$$
f(z)=g(z+1)-g(z).
$$
Then \eqref{aux524} holds on every tie and every left rail of every track with horizontal ties. To see this, note that the above equality is a special case of \eqref{aux524} with
$u=z,$ $v=z+1.$ 
Furthermore, if $abcd$ is a face of $\Lambda,$ such that $c=b+1$ and $d=a+1,$ then 
\begin{multline*}
(1+b-a)f(a)-(1+a-b)f(b)\\=(g(a+1)-g(a))(1+b-a)-(g(b+1)-g(b))(1+a-b)\\
=(g(a+1)-g(b))(b+1-a)-(g(b+1)-g(a))(a+1-b)+\\+2(g(b)-g(a)).
\end{multline*}
By Theorem \ref{CR},
$$\dfrac{g(a+1)-g(b)}{a+1-b}=\dfrac{g(b+1)-g(a)}{b+1-a},$$
hence \eqref{aux524}  holds on the edge $(a,b),$ as well.

\noindent\paragraph{\bf Step 3.} If $abcd$ is a face without horizontal edges, and if the value $f(a)$ is given, then there is a unique way to assign the values $f(b),f(c),f(d)$ so that \eqref{aux524} holds on  all the edges of the face $abcd.$ Indeed, \eqref{aux524} holds on the edges $(a,b),$ $(b,c)$ and $(c,d)$ if, and only if,
\begin{gather*}
f(b)=\dfrac{2(g(b)-g(a))-(b-a+1)f(a)}{b-a-1},\\
f(c)=\dfrac{2(g(c)-g(b))-(c-b+1)f(b)}{c-b-1},\\
f(d)=\dfrac{2(g(d)-g(c))-(d-c+1)f(c)}{d-c-1}.
\end{gather*}
In this case one has
\begin{multline*}2(g(c)-g(a))=2(g(c)-g(b))+2(g(b)-g(a))\\
=(b-a+1)f(a)+(c-a)f(b)+(c-b-1)f(c)
\end{multline*}
and, similarly, 
$$2(g(d)-g(b))=(c-b+1)f(b)+(d-b)f(c)+(d-c-1)f(d).$$
Since $g(z)$ is a DA function,  it follows from Theorem \ref{CR} that
\begin{gather*}
\begin{split}
\dfrac{b-a+1}{c-a}f(a)&+f(b)+\dfrac{c-b-1}{c-a}f(c)\\&=\dfrac{c-b+1}{d-b}f(b)+f(c)+\dfrac{d-c-1}{d-b}f(d),\end{split}\\
(a-b-1)\dfrac{f(c)-f(a)}{c-a}-(d-c-1)\dfrac{f(d)-f(b)}{d-b}=0.
\end{gather*}
Since $a-b=d-c\not=1,$ one has
$$\dfrac{f(c)-f(a)}{c-a}=\dfrac{f(d)-f(b)}{d-b}$$
and
$$\int_{abcd}f\delta z=0.$$ But then
\begin{multline*} 
g(d)-g(a)=g(d)-g(c)+g(c)-g(b)+g(b)-g(a)\\
=\dfrac{f(a)-f(d)}{2}+\int_{abcd}f\delta z-\dfrac{f(a)+f(d)}{2}(a-d)\\
=\dfrac{f(a)-f(d)}{2}+\dfrac{f(a)+f(d}{2}(d-a),
\end{multline*}
and equality \eqref{aux524}  on the edge $(a,d)$ follows. 

\noindent\paragraph{\bf Step 4.} Proceeding as at Step 3, one can define $f(z)$  on the vertices of any face without horizontal edges that has a vertex with a leash of length 1. This set includes, in particular, all vertices  of $\Lambda$ that have a leash of length at most $2.$ By induction on the length of the shortest leash, one can extend the definition of $f(z)$ to the whole of $V(\Lambda)$ in such a way that \eqref{aux524} holds on every edge of $\Lambda.$ 
\end{proof}

\begin{definition}\label{bwd} The backward shift operator $Z_-:H(\Lambda)\longrightarrow H(\Lambda)$ is defined as follows:
given $f\in\cH(\Lambda),$ $Z_-f$ is the unique  element of $\cH(\Lambda),$ such that
$$Z_+Z_- f(z)=f(z)-f(0).$$
\end{definition}

\begin{proposition}\label{bwdprop} Operator $Z_-$ is a left inverse of $Z_+,$ and the kernel of $Z_-$ consists of constant DA functions.
\end{proposition}
\begin{proof}
Let $f\in\cH(\Lambda).$ In view of Theorem \ref{fwdthm} and Definition \ref{bwd}, one has
\begin{gather*}Z_+Z_-Z_+f(z)=Z_+f(z)-Z_+f(0)=Z_+f(z),\\
Z_-Z_+f(z)=f(z).\end{gather*} As to the kernel of $Z_-,$ it follows immediately from Definition \ref{bwd}
that
$$Z_-f(z)=0\quad \Leftrightarrow\quad f(z)-f(0)=0\quad \Leftrightarrow\quad f(z)=f(0).$$
\end{proof}

\begin{remark} \label{shiftrem} The shift operators $Z_+$ and $Z_-$ are "discrete counterparts" of the classical shift operators
$$f(z)\mapsto zf(z)\quad\text{and}\quad f(z)\mapsto\dfrac{f(z)-f(0)}{z}$$
on the space of continuous analytic functions in the open unit disk.
\end{remark}

\section{Discrete analytic polynomials}
\label{DAP}

Denote  $$\bS(\Lambda)=\{t\in\CC : t(1+b-a)=-2\text{ for some }(a,b)\in \bE(\Lambda)\}.$$ 
In view of Proposition \ref{finder},  $\bS(\Lambda)$ is a finite set. It is contained in the complement of the open unit disk.
\begin{theorem}\label{genfunthm} Let $t\in\CC\setminus\bS(\Lambda).$ 
Then $t$ is an eigenvalue of the operator $Z_-.$ The corresponding eigenspace is one-dimensional; it is spanned by the DA function $e_t(z),$ which is determined as follows: $e_t(0)=1,$ and
\begin{equation}\label{genfun}e_t(z)=
\prod_{k=1}^N\dfrac{2+t(1+z_k-z_{k-1})}{2+t(1+z_{k-1}-z_k)}\end{equation}
for $z\not=0,$ where $(z_0,z_1,\dots z_N)$ is any path from $z_0=0$ to $z_N=z.$
\end{theorem}

\begin{proof}
The proof is similar to that of Theorem \ref{fwdthm}. It can be broken down into a number of steps. 

\noindent\paragraph{\bf Step 1.} Let $t\in\CC\setminus\bS(\Lambda),$  and let  $e_t:V(\Lambda)\longrightarrow\CC.$ 
Then $e_t\in\cH(\Lambda)$ and $Z_-e_t=te_t$ if, and only if, 
the relation
\begin{equation}\label{aux1202} e_t(u)(2+t(1+v-u))=e_t(v)(2+t(1+u-v))\end{equation}
holds on every edge $(u,v)$ of $\Lambda.$ Indeed, if $e_t(z)$ is a DA function, then, by Definition \ref{bwd} and Proposition \ref{bwdprop},  equality
$$Z_-e_t(z)=te_t(z),$$
is equivalent to
 $$e_t(z)-e_t(0)=tZ_+e_t(z).$$
 This last equality holds if, and only if, 
 for every pair of adjacent vertices $u$ and $v$ one has
\begin{multline*} e_t(u)-e_t(v)=t(Z_+e_t(u)-Z_+e_t(v))=t\left(\dfrac{e_t(v)-e_t(u)}{2}+\int_v^u e_t\delta z\right)\\=
t\left(\dfrac{e_t(v)-e_t(u)}{2}+\dfrac{e_t(v)+e_t(u)}{2}(u-v)\right),\end{multline*}
which is equivalent to \eqref{aux1202}. Note, however, that if a function $e_t(z)$ satisfies \eqref{aux1202} on every edge of $\Lambda,$ then it is automatically DA, since  on every face $abcd$ 
one then has
\begin{multline*}
\dfrac{e_t(b)-e_t(d)}{b-d}=\dfrac{e_t(a)}{b-d}\left(\dfrac{2+t(1+b-a)}{2+t(1+a-b)}-\dfrac{2+t(1+d-a)}{2+t(1+a-d)}\right)\\
=\dfrac{e_t(a)(4t+2t^2)}{(2+t(1+a-b))(2+t(1+a-d))}\\=\dfrac{e_t(a)(4t+2t^2)}{(2+t(1+a-b))(2+t(1+b-c))}\\=
\dfrac{e_t(a)}{c-a}\left(\dfrac{(2+t(1+b-a))(2+t(1+c-b))}{(2+t(1+a-b))(2+t(1+b-c))}-1\right)=\dfrac{e_t(c)-e_t(a)}{c-a}.  \end{multline*}

\noindent\paragraph{\bf Step 2.}  Since $t\not\in\bS(\Lambda),$  one can define a function $e_t(z)$ by $e_t(0)=1$ and \eqref{genfun} for $z\not=0,$ where $(z_0,z_1,\dots z_N)$ is some path from $z_0=0$ to $z_N=z.$ Note that definition \eqref{genfun} is actually path-independent because on every face $abcd$ of $\Lambda$
one has:
\begin{gather*}
b-a=c-d,\quad c-b=d-a,\\
\begin{split}\dfrac{(2+t(1+b-a))(2+t(1+c-b))}{(2+t(1+a-b))(2+t(1+b-c))}\times\qquad\ \\
\times\dfrac{(2+t(1+d-c))(2+t(1+a-d))}{(2+t(1+c-d))(2+t(1+d-a))}&=1.\end{split}
\end{gather*}
Therefore, $e_t(z)$ satisfies \eqref{aux1202} on every edge of $\Lambda$ and is an eigenfunction of $Z_-$ associated with the eigenvalue $t.$ It is also clear from \eqref{aux1202} that such a function is unique up to a multiplicative constant.
\end{proof}

\begin{remark} \label{r0rem}Denote by $\bS_0(\Lambda)$  the complement of  the set of eigenvalues of the operator $Z_-.$  Using an approach that is quite similar to the proof of Theorem \ref{fwdthm}, one can show that $t=-1$ is not an eigenvalue of $Z_-,$ and hence $\bS_0(\Lambda)\not=\emptyset.$ On the other hand, according to Theorem \ref{genfunthm},  $\bS_0(\Lambda)\subset\bS(\Lambda).$ Depending on the lattice $\Lambda,$ the inclusion can be proper, although in some important cases, including the case of the square lattice, it is not. It is assumed from now on that $\bS_0(\Lambda)=\bS(\Lambda).$ \end{remark}

In view of \eqref{genfun}, $e_t(z)$ is a rational function in $t$ (of degree depending on $z,$ in general), analytic in the complement of $\bS(\Lambda)$ and, in particular, in the open unit disk $\DD.$ Consider the Taylor expansion
\begin{equation}\label{genfunexp}e_t(z)=\sum_{n=0}^\infty t^n z^{(n)},\quad t\in\DD, z\in \bV(\Lambda).\end{equation}
\begin{proposition}\label{tayprop} Taylor coefficients $z^{(n)}$ in the expansion \eqref{genfunexp} are DA functions of $z$ with the following properties:
\begin{equation}
\label{fwdbas}z^{(0)}=1\quad\text{and}\quad z^{(n)}=Z_+z^{(n-1)}\quad\forall n\in\NN,\end{equation}
in particular,  $z^{(1)}=z,$
\begin{gather}
\label{bwdbas} Z_-z^{(0)}=0\quad\text{and}\quad Z_-z^{(n)}=z^{(n-1)}\quad\forall n\in\NN,\\
\label{cauchy}\limsup_{n\rightarrow\infty}\sqrt[n]{|z^{(n)}|}\leq 1.
\end{gather}
Moreover, $z^{(0)},z^{(1)},z^{(2)},\dots$ form a linearly independent vector family in $\cH(\Lambda).$
\end{proposition}
\begin{proof} Discrete analyticity of $z^{(n)}$ can be seen by differentiating the discrete Cauchy - Riemann equation for $e_t(z)$ with respect to $t$.  Substituting the expansion \eqref{genfunexp} into the equality
$$e_t(z)-1=tZ_+e_t(z),$$ one obtains \eqref{fwdbas}. In particular,
$$z^{(1)}=Z_+1=\int_0^z \delta z=z.$$
In view of Proposition \ref{bwdprop}, \eqref{bwdbas} holds, as well, and hence for a linear combination $a_0z^{(0)}+a_1z^{(1)}+\dots+a_Nz^{(N)}$ one has
$$(Z_-^n(a_0z^{(0)}+a_1z^{(1)}+\dots+a_Nz^{(N)}))(0)=a_n,\quad n=0,1,\dots,N,$$
implying  linear independence of the family $z^{(0)},z^{(1)},z^{(2)},\dots$ Finally,  formula \eqref{cauchy} is a consequence of the Cauchy - Hadamard theorem.\end{proof}

\begin{definition} \label{dapoly} Elements of the subspace of $\cH(\Lambda)$ spanned by $$z^{(0)},z^{(1)},z^{(2)},\dots$$ are called DA polynomials. 
\end{definition}

It should be mentioned here that the space of DA polynomials defined above was originally introduced by Duffin in \cite{rhombic} using a different basis:
$$\rho_0=1,\quad \rho_n(z)=n\int_0^z \rho_{n-1}\delta z.$$

\section{Convolution product and rational DA functions}
\label{ratsec}
There is little difficulty involved in adapting the results of previous sections to the case of DA functions with matricial values. Such functions can be viewed as matrices with DA entries, with the discrete integral and the shift operators being applied entry-wise. The space of $\CC^{m\times n}$-valued DA functions on $\bV(\Lambda)$ is denoted by $\cH(\Lambda)^{m\times n}.$ Matrix-valued DA polynomials are introduced similarly.

\begin{definition}\label{convoproddef}
The convolution product $\odot$ of a $\CC^{m\times n}$-valued DA polynomial 
$$p(z)=\sum_{n=0}^N A_nz^{(n)}$$
with a  $\CC^{n\times k}$-valued DA function $f(z)$  is given by
$$(p\odot f)(z):=\sum_{n=0}^N A_n(Z_+^nf)(z).$$
Similarly,  if $g(z)$ is a $\CC^{k\times m}$-valued DA function, then 
$$(g\odot p)(z):=\sum_{n=0}^N (Z_+^ng)(z)A_n.$$
\end{definition}

Note that the space of $\CC$-valued DA polynomials equipped with the convolution product $\odot$
is a commutative ring, where
$$z^{(m)}\odot z^{(n)}=z^{(m+n)}.$$

\begin{theorem}\label{resolve} Let $A\in\CC^{m\times m}.$  Then the DA polynomial $I_m-zA$ is  $\odot$-invertible in $\cH(\Lambda)^{m\times m}$
if,  and only if, 
\begin{equation}\label{excl}\bS(\Lambda)\cap\sigma(A)=\emptyset.\end{equation}
In this case the  $\odot$-inverse  of   $I_m-zA$  is given by
\begin{equation}\label{resA}(I_m-zA)^{-\odot}:=\prod_{k=1}^N\big((2I_m+(1+z_k-z_{k-1})A)(2I_m+(1+z_{k-1}-z_k)A)^{-1}\big),\end{equation}
where $(z_0,z_1,\dots z_N)$ is any path from $z_0=0$ to $z_N=z.$
\end{theorem}
\begin{proof}  Let $r(z)=I_m-zA$ and 
suppose  that $f$ is  a   left $\odot$-inverse of $r(z)$ in  $\cH(\Lambda)^{m\times m}:$ 
$$(f\odot r)(z)=I_m,$$ or, equivalently,
\begin{equation}\label{lid}f(z)=I_m+Z_+f(z)A.\end{equation}
Then $$f(0)=I_m\quad\text{and}\quad Z_-f(z)=f(z)A.$$
Let $t$ be an eigenvalue of $A$ with the associated eigenvector $w\not=0.$
Then $$Z_-f(z)w=tf(z)w.$$
If $t\in\bS,$ then (see Remark \ref{r0rem}) $f(z)w=0,$ which leads to a contradiction for $z=0.$ 
Thus the necessity of the condition \eqref{excl} is established.

Conversely, suppose that \eqref{excl} is in force and observe that \eqref{lid} is equivalent to $f(0)=I_m$ and
$$
f(u)(2I_m+(1+v-u)A)=f(v)(2I_m+(1+u-v)A)
$$
holding on every edge of $\Lambda$ (compare with \eqref{aux1202}). Exactly as in the proof of Theorem \ref{genfunthm}, one may deduce
that there  exists  a unique DA function $f(z)$ satisfying \eqref{lid}. Moreover,  $f(z)=(I_m-zA)^{-\odot}$ is given by  \eqref{resA} (in particular, this formula is path-independent).  Since $(I_m-zA)^{-\odot}$ commutes with $A,$ it is a right $\odot$-inverse of $r(z),$ as well. 
\end{proof}

\begin{remark}
If the spectral radius $\rho(A)<1,$ then
$$(I-zA)^{-\odot}=\sum_{n=0}^\infty z^{(n)}A^n.$$
\end{remark}

\begin{definition}\label{defrat} A matrix-valued DA function $f(z)$ is said to be rational if it can be represented as
$$f(z)=D+C(I-zA)^{-\odot}\odot (zB),$$
where $A,B,C,D$ are complex matrices of suitable dimensions,  and 
$$\bS(\Lambda)\cap\sigma(A)=\emptyset.$$ Such a representation is called a realization of $f(z).$
\end{definition}

It may not be  obvious from the above definition that a sum of matrix-valued rational DA functions of suitable dimensions is itself rational.
This is settled in the proposition below.

\begin{proposition} Let $m,n\in\NN.$
Let $f_1,f_2\in\cH(\Lambda)^{m\times n}$ be rational DA functions with given realizations:
$$f_j(z)=D_j+C_j(I-zA_j)^{-\odot}\odot(zB_j),\quad j=1,2.$$
Then $f_1(z) +f_2(z)$ is a rational DA function, as well, which admits the realization
$$(f_1+f_2)(z)
=D_1+D_2+\begin{pmatrix}C_1&C_2\end{pmatrix}\left(I-z\begin{pmatrix}A_1&0\\0&A_2\end{pmatrix}\right)^{-\odot}\odot
\left(z\begin{pmatrix}B_1\\B_2\end{pmatrix}\right).
$$
\end{proposition}
\begin{proof}
It suffices to observe that
$$\left(I-z\begin{pmatrix}A_1&0\\0&A_2\end{pmatrix}\right)^{-\odot}=\begin{pmatrix}(I-zA_1)^{-\odot}&0\\0&(I-zA_2)^{-\odot}\end{pmatrix}.$$\end{proof}

Definition \ref{defrat} suggests that rational DA functions are closely related to rational functions of the complex variable. One must treat this relationship with caution, since
a given rational function admits infinitely many distinct realizations. A precise formulation of the relationship is stated in the following theorem.

\begin{theorem} \label{ratmap} Let $m,n\in\NN.$
The mapping 
$$f(z)\mapsto \tau f(t):=\sum_{k=0}^\infty Z_-^k f(0)t^k$$
is a linear bijection from the space of $\CC^{m\times n}$-valued rational DA functions to the space of 
$\CC^{m\times n}$-valued rational  functions of the complex variable $t$   that have no poles in the set 
$$\bP(\Lambda)=\left\{t\in\CC: t=0\text{ or }\dfrac{1}{t}\in\bS(\Lambda)\right\}.$$
For any realization 
\begin{equation}\label{realization}f(z)=D+C(I-zA)^{-\odot}\odot (zB)\end{equation}
 of a rational DA function $f(z)$ it holds that
\begin{equation}\label{realreal}\tau f(t)=D+tC(I-tA)^{-1}B.\end{equation}
\end{theorem}
\begin{proof}
From  the identity
$$(I-zA)\odot(I-zA)^{-\odot}=I$$ and
Proposition \ref{bwdprop} it follows that
$$Z_-(I-zA)^{-\odot}=(I-zA)^{-\odot}A.$$
Therefore,
for a rational DA function $f(z)$ of the form \eqref{realization} one has
$$D=f(0),\quad\text{and}\quad Z_-^kf(0)=CA^{k-1}B\quad\forall k\in\NN, $$ from which \eqref{realreal} and  the surjectivity of the mapping $\tau$ follow.

To verify the injectivity of $\tau,$ consider a rational function
$$D+C(I-zA)^{-\odot}\odot(zB)$$
in the kernel of $\tau,$ where $A\in\CC^{\ell\times \ell}$  satisfies  \eqref{excl}. Then
$$D=0\text{ and } CA^{k-1}B=0\quad \forall k\in\NN.$$ Consider
$$\cM=\ran(B)+\ran(AB)+\ran(A^2B)+\dots+\ran(A^{\ell-1}B)$$
-- an $A$-invariant subspace of the finite-dimensional $\ker(C)$. Since 
  $A$ satisfies  \eqref{excl},  for every edge $(u,v)$ of $\Lambda$ the matrix $2I_\ell+(1+u-v)A$ is non-degenerate, and the associated linear transformation of $\CC^\ell$ maps  $\cM$ bijectively to itself.  Therefore, In view of \eqref{resA},  
$$C(I-zA)^{-\odot}B=0$$
and
$$C(I-zA)^{-\odot}\odot(zB)=Z(C(I-zA)^{-\odot}B)=0.$$
\end{proof}

Theorem \ref{ratmap} allows to extend the convolution product $\odot$ to the case where the factors are rational DA functions.

\begin{definition}\label{defprod}
Let $m,n,k\in\NN.$ Let $f_1\in\cH(\Lambda)^{m\times n}$ and $f_2\in\cH(\Lambda)^{n\times k}$ be rational DA functions. Define the convolution product $f_1\odot f_2$ by
$$f_1\odot f_2=\tau^{-1}(\tau f_1\cdot \tau f_2).$$
\end{definition}

\begin{remark} Note that in the case where $f_1$ or $f_2$ is a DA polynomial, Definition \ref{defprod} agrees with Definition \ref{convoproddef}, since
$$\tau z^{(n)}= t^n.$$
\end{remark}

In terms of realizations, the convolution product of rational DA functions can be described more explicitly, as the following proposition shows.

\begin{proposition}\label{prodprop} Let $f_1\in\cH(\Lambda)^{m\times n}$ and $f_2\in\cH(\Lambda)^{n\times k}$ be rational DA functions with given realizations:
$$f_j(z)=D_j+C_j(I-zA_j)^{-\odot}\odot(zB_j),\quad j=1,2,$$
then
$$(f_1\odot f_2)(z)=D+C(I-zA)^{-\odot}\odot(zB),$$
where
\begin{gather*}
A=\begin{pmatrix} A_2&0\\ B_1C_2&A_1\end{pmatrix}, \quad B=\begin{pmatrix} B_2\\ B_1D_2\end{pmatrix};\\
C=\begin{pmatrix} D_1C_2&C_1\end{pmatrix},\quad D=D_1D_2.
\end{gather*}
\end{proposition}
\begin{proof} In view of Theorem \ref{ratmap}, it suffices to check that the rational matrix-valued function $\tau f_1(t)\tau f_2(t)$ admits the realization
$$\tau f_1(t)\tau f_2(t)=D+tC(I-tA)^{-1}B,$$
where $A,B,C,D$ are given above. This is straightforward,
since 
\begin{multline*}(I-tA)^{-1}=\begin{pmatrix} I-tA_2&0\\-tB_1C_2&I-tA_1\end{pmatrix}^{-1}\\[1ex]=\begin{pmatrix}(I-tA_2)^{-1}&0\\t(I-tA_1)^{-1}B_1C_2(I-tA_2)^{-1}& (I-tA_1)^{-1}\end{pmatrix}.\end{multline*}
\end{proof}

The question of $\odot$-invertibility is to be addressed next. As follows from Definition \ref{defprod}, a matrix-valued rational DA function $f(z)$ has a rational $\odot$-inverse if, and only if, the rational matrix-function $\tau f(t)$ is invertible and $\tau f(t)^{-1}$ has no poles in the set $\bP(\Lambda)$ of Theorem \ref{ratmap}. In this case such a $\odot$-inverse is unique:
$$f^{-\odot}(z)=\tau^{-1}((\tau f)(t)^{-1}).$$ 
A more explicit expression for the $\odot$-inverse in terms of realizations is given in  the next proposition.

\begin{proposition}\label{invprop}
Let $n\in\NN.$ Let $f\in\cH(\Lambda)^{n\times n}$ be a rational DA function
 with a realization
$$f(z)=D+C(I-zA)^{-\odot}\odot (zB).$$
Denote $$A^\times=A-BD^{-1}C.$$ If $$\det(D)\not=0\quad\text{and}\quad \bS(\Lambda)\cap\sigma(A^\times)=\emptyset,$$
then $f(z)$ has a unique rational $\odot$-inverse:
$$f^{-\odot}(z)=D^{-1}-D^{-1}C(I-zA^\times)^{-\odot}\odot(zBD^{-1}).$$
\end{proposition}
\begin{proof}
The statement  follows from Proposition \ref{prodprop} by straightforward computation.
\end{proof}

The following theorem is the "discrete counterpart" of the classical definition of rationality in terms of quotients of polynomials.
\begin{theorem} \label{main} Let $f(z)$ be a matrix-valued DA function. Function $f(z)$ is rational if,  and only if,  there exists a $\CC$-valued  DA polynomial $p(z)\not\equiv 0,$ such that
$(pI\odot f)(z)$ is a DA polynomial, as well.
\end{theorem}
\begin{proof}
The "only if" part of the statement follows  immediately from Theorem \ref{ratmap}, Definition \ref{defprod} and the fact that $\tau$ maps DA polynomials to ordinary polynomials of one complex variable. As to the "if" part,   let $p(z)\not\equiv 0$ be a $\CC$-valued  DA polynomial, of the minimal possible degree, such that $(pI\odot f)(z)$
is a DA polynomial. It suffices to show that $p(z)$ has a rational $\odot$-inverse.

Suppose, to obtain a contradiction, that there exists $\lambda\in\bP(\Lambda),$ such that $\tau p(\lambda)=0.$ 
If $\lambda=0,$ then
 $p(z)=z\odot p_1(z),$ where 
 $$p_1(z)=\tau^{-1}\left(\dfrac{\tau p(t)}{t}\right)=Z_- p(z)$$ is a DA polynomial of degree $\deg(p_1)=\deg(p)-1.$
Hence  $$p_1I\odot f=Z_-(pI\odot f)$$
is also a DA polynomial, in contradiction to the minimality of $\deg(p).$ 

If $\lambda\not=0,$ then $\frac{1}{\lambda}\in\bS(\Lambda)$ and $p(z)=(z-\lambda)\odot p_1(z),$ where 
 $$p_1(z)=\tau^{-1}\left(\dfrac{\tau p(t)}{t-\lambda}\right)$$ is a DA polynomial of degree $\deg(p_1)=\deg(p)-1.$
Hence  $$pI\odot f=(Z_+-\lambda I)(p_1I\odot f)$$
and
$$\left(\dfrac{1}{\lambda}I-Z_-\right)(p_1I\odot f)=\dfrac{1}{\lambda}Z_-(pI\odot f)$$
is also a DA polynomial. Note that for every $N\in\NN$  the finite-dimensional
space $\cM_N$ of DA polynomials of degree not exceeding $N$ is invariant  under $\frac{1}{\lambda}I-Z_-.$ Moreover, since $\frac{1}{\lambda}$ is not an eigenvalue of $Z_-$ (see Remark \ref{r0rem}), the operator $\frac{1}{\lambda}I-Z_-$ maps $\cM_N$ onto itself. Therefore, $p_1I\odot f$ is a DA polynomial, yet again in contradiction to the minimality of $\deg(p).$
 
Thus $\tau p(t)$ does not vanish on the set $\bP(\Lambda).$ But then the rational function $\frac{1}{\tau p(t)}$ has no poles in $\bP(\Lambda),$ the polynomial $p(z)$ has a rational $\odot$-inverse, and, finally, $$f=(p^{-\odot}I)\odot(pI\odot f)$$
is a rational DA function.
\end{proof}
 
 Another useful characterization of DA rationality is in terms of  shift invariance. For a rational DA function $f(z),$
 $$Z_+f(z)=z\odot f(z),$$ hence the space of matrix-valued rational DA functions (of given dimensions) is forward shift invariant.
 It is also backward shift invariant, since for
 $$f(z)=D+C(I-zA)^{-\odot}\odot(zB)$$
 one has
 \begin{multline*} Z_-f(z)=C(I-zA)^{-\odot}B=CB+C(Z_+Z_-(I-zA)^{-\odot})B\\=CB+C(I-zA)^{-\odot}\odot(zAB).\end{multline*}
 
\begin{theorem} \label{shiftcar} Let $m,n\in\NN,$ and let  $f\in\cH^{m\times n}.$ DA function $f(z)$ is rational if,  and only if,  
the column space of $f(z), $ $\{f(z)w:w\in\CC^n\},$ is contained in a finite-dimensional space of $\CC^m$-valued DA functions that is $Z_-$-invariant.
\end{theorem}
\begin{proof}
The "only if" part of the statement follows  from the Caley - Hamilton theorem, since for
 $$f(z)=D+C(I-zA)^{-\odot}\odot(zB)$$
 one has 
$$Z_-^kf(z)=C(I-zA)^{-\odot}A^{k-1}B\quad\forall k\in\NN$$ and, therefore,
$$\dim\spa_\CC\{Z_-^kf(z)w:k\in\ZZ_+,w\in\CC^n\}<\infty.$$

As to the "if" part,
let $\cM$ be a finite-dimensional subspace of $\cH(\Lambda)^m,$ such that
$$\{ f(z)w:w\in\CC^n\}\subset \cM\quad\text{and}\quad Z_-\cM\subset\cM.$$
Choose a basis $f_1(z), f_2(z),\dots, f_N(z)$ of $\cM$ and consider
the 
 $\CC^{m\times N}$-valued DA function
$$F(z):=\begin{pmatrix} f_1(z)&f_2(z)&\cdots& f_N(z)\end{pmatrix},.$$
Then there exists a matrix $A\in\CC^{N\times N},$ such that
$$Z_- F(z)=F(z)A$$
Suppose, to obtain a contradiction, that $A$ has an eigenvalue $t\in\bS(\Lambda).$
Then 
$$(tI-Z_-)F(z)=F(z)(tI_{N}- A).$$
 If $w$ is an eigenvector of $A$ associated with the eigenvalue $t,$ then   $F(z)w=0$ because $t$ is not an eigenvalue of $Z_-$ (se Remark \ref{r0rem}). This contradicts the linear independence of the columns of $F(z).$ Thus $A$ satisfies \eqref{excl}. Furthermore, one has
 \begin{gather*}F(z)-F(0)=Z_+Z_-F(z)=F(z)\odot(zA),\\
 F(z)\odot(I-zA)=F(0),\\
 F(z)=C(I-zA)^{-\odot},
 \end{gather*}
 where $C=F(0).$ Since the column space of $f(z)$ is contained in $\cM,$ there exists a matrix $B\in\CC{N\times n},$ such that
 $$f(z)=F(z)B=C(I-zA)^{-\odot}B=CB+C(I-zA)^{-\odot}\odot(zAB).$$
 \end{proof}
It remains to provide a non-trivial example of a rational DA function.
\begin{example}
  \label{rk-ex}
 Let $w\in\bV(\Lambda)$ be fixed and let $M>1.$ Consider the DA function 
 $$K_w(z)=\sum_{n=0}^\infty \dfrac{z^{(n)}\overline{w^{(n)}}}{M^n}$$
 (in view of \eqref{cauchy}, the series converges absolutely for every $z\in\bV(\Lambda)$ to a function that is DA by Theorem \ref{CR}). 
 If $K_w(z)$ were a rational DA function, then
 $$\tau K_w(t)=\sum_{n=0}^\infty \dfrac{t^n\overline{w^{(n)}}}{M^n}=\overline{e_{\bar t/M}(w)},$$
 which, in view of \eqref{genfun}, is a rational function of the complex variable $t$ and has no poles in $\bP(\Lambda).$
Taking into account the fact that $e_0(w)=1,$ one may conclude that there exist matrices $A,B,C$ satisfying \eqref{excl} and such that
 $$\overline{e_{\bar t/M}(w)}=1+tC(I-tA)^{-1}B.$$
In particular,
$$\dfrac{\overline{w^{(n)}}}{M^n}=CA^{n-1}B\quad \forall n\in\NN.$$
Hence, rational or not, function $K_w(z)$ can be represented as
$$K_w(z)=K_w(z)=1+\sum_{n=1}^\infty z^{(n)}CA^{n-1}B.$$
Observe that
$$Z_-^mK_w(z)=\sum_{n=0}^\infty z^{(n)}CA^{n+m-1}B\quad\forall m\in\NN.$$
By the Caley-Hamilton theorem, there is $N\in\NN$ and $a_0,\dots, a_{N-1}\in\CC$ such that
$$A^N=\sum_{j=0}^{N-1}a_jA^j,$$
hence
$$Z_-^{N+1}K_w(z)=\sum_{j=1}^{N}a_ {j-1}Z_-^jK_w(z).$$
Let 
$$\cM=\spa_\CC\{Z_-^jK_w(z):j=0,1,\dots,N\}.$$
Then $\cM$ is a $Z_-$-invariant subspace of $\cH(\Lambda)$ of dimension at most $N+1$ that contains $K_w(z).$  By Theorem \ref{shiftcar},  $K_w(z)$ is indeed a  rational DA function,
$$K_w(z)=\tau^{-1}\left(\overline{e_{\bar t/m}(w)}\right).$$
\end{example}

\section*{Acknowledgments}
Daniel Alpay thanks the Foster G. and Mary McGaw Professorship in Mathematical Sciences, which supported this research. Zubayir Kazi, Mariana Tecalero, and Dan Volok thank National Science Foundation for support (grant \# 2243854) and Kansas State University for hosting the participants in its Mathematics REU (SUMaR) in the summer of 2023.

\bibliographystyle{plain}
%\bibliography{all}
%\def\cprime{$'$} \def\cprime{$'$} \def\cprime{$'$}
  %\def\lfhook#1{\setbox0=\hbox{#1}{\ooalign{\hidewidth
 % \lower1.5ex\hbox{'}\hidewidth\crcr\unhbox0}}} \def\cprime{$'$}
 % \def\cprime{$'$} \def\cprime{$'$} \def\cprime{$'$} \def\cprime{$'$}
  %\def\cprime{$'$}

\end{document}